\documentclass[11pt, reqno]{amsart}

\usepackage{geometry}
\usepackage{color}
\usepackage{amssymb}
\usepackage{amsmath}
\usepackage{arydshln}
\usepackage{hyperref}
\usepackage{bbm}
\usepackage{mathrsfs}
\usepackage{enumerate}

\usepackage{graphicx}
\usepackage{subfig}

\numberwithin{equation}{section}

\newcommand{\R}{\mathbb{R}}

\newcommand{\N}{\mathbb{N}}

\renewcommand{\P}{\mathbb{P}}
\newcommand{\E}{\mathbb{E}}

\newcommand{\F}{\mathcal{F}}
\newcommand{\cF}{\mathcal{F}}

\newcommand{\e}{\varepsilon}
\newcommand{\1}{\mathbbm{1}}

\newtheorem{Theorem}{Theorem}[section]
\newtheorem{Proposition}[Theorem]{Proposition}
\newtheorem{Corollary}[Theorem]{Corollary}
\newtheorem{Lemma}[Theorem]{Lemma}
\newtheorem{Remark}[Theorem]{Remark}

\newtheorem{Example}[Theorem]{Example}

\begin{document}

\title[Limit theorems for time averages of CBI processes]{Limit theorems for time averages of continuous-state branching processes with immigration}

\author{Mariem Abdellatif}
\address[Mariem Abdellatif]{School of Mathematics and Natural Sciences\\ University of Wuppertal, Germany}
\email[Mariem Abdellatif]{abdellatif@uni-wuppertal.de}

\author{Martin Friesen}
\address[Martin Friesen]{School of Mathematical Sciences\\
Dublin City University\\ Glasnevin, Dublin 9, Ireland}
\email[Martin Friesen]{martin.friesen@dcu.ie}

\author{Peter Kuchling}
\address[Peter Kuchling]{School of Mathematics and Natural Sciences\\ University of Wuppertal, Germany}
\email[Peter Kuchling]{kuchling@uni-wuppertal.de}

\author{Barbara R\"udiger}
\address[Barbara R\"udiger]{School of Mathematics and Natural Sciences\\ University of Wuppertal, Germany}
\email[Barbara R\"udiger]{ruediger@uni-wuppertal.de}

\date{\today}


\begin{abstract}
 In this work we investigate limit theorems for the time-averaged process $\left( \frac{1}{t}\int_0^t X_s^x ds \right)_{t \geq 0}$ where $X^x$ is a subcritical continuous-state branching process with immigration starting in $x \geq 0$. Under a second moment condition on the branching and immigration measures we first prove the law of large numbers in $L^2$ and afterward establish the central limit theorem. Assuming additionally that the big jumps of the branching and immigration measures have finite exponential moments of some order, we prove in our main result the large deviation principle and provide a semi-explicit expression for the good rate function in terms of the branching and immigration mechanisms. Our methods are deeply based on a detailed study of the corresponding generalized Riccati equation and related exponential moments of the time-averaged process. 
\end{abstract}

\maketitle

\allowdisplaybreaks

\section{Introduction}

\subsection{General introduction}
One-dimensional continuous-state branching processes with immigration (CBI processes) form an important class of Markov processes on the state space $\R_+=[0,\infty)$. Such processes have been first introduced and studied by Ji\v{r}ina \cite{jirina}, Lamperti \cite{lamperti}, Silberstein \cite{silberstein} and Watanabe \cite{watanabe}. For an up-to-date overview and additional references on CBI processes we refer to Li \cite{li_lecture, li_book}. Since CBI processes form a particular class of affine processes, they are nowadays widely used in Mathematical Finance for the modelling of default intensities, stochastic interest rates, and stochastic volatility (see \cite{dfs, MR3363174, JMS17}). In mathematical terms, a CBI process on $\R_+$ is a time-homogeneous Markov process whose transition probability kernel $p_t(x,dy)$ satisfies for $x,t \geq 0$ the affine-transformation formula
\begin{align}\label{eq: affine formula}
 \int_{\R_+} e^{\lambda z} p_t(x,dz) = \exp\left(x v(t,\lambda) + \int_0^t F(v(s,\lambda))ds \right), \qquad \lambda \leq 0.
\end{align}
Here $v$ denotes the unique solution of the generalized Riccati equation
\[
 v'(t,\lambda) = R(v(t,\lambda)), \qquad v(0,\lambda) = \lambda
\]
while $F,R$ are of L\'evy-Khinchine form
\begin{align*}
 F(u) &= bu + \int_0^{\infty}\left( e^{u z} - 1 \right)\nu(dz),
 \\ R(u) &= \beta u + \frac{\sigma^2}{2}u^2 + \int_0^{\infty}\left( e^{u z} - 1 - uz \right)\mu(dz)
\end{align*}
with $b, \sigma \geq 0$, $\beta \in \R$, and $\nu,\mu$ are $\sigma$-finite measures on $[0,\infty)$ with $\nu(\{0\}) = \mu(\{0\}) = 0$ satisfying the integrability conditions
\[
 \int_0^{\infty}(1 \wedge z) \nu(dz) + \int_0^{\infty}(z \wedge z^2)\mu(dz) < \infty.
\]
It can be shown that such a CBI process is conservative and that its transition semigroup obtained from \eqref{eq: affine formula} is actually $C_0$-Feller, see e.g. \cite{kawazu_watanabe}. 

Following \cite[Chapter 8]{li_lecture} (see also the references therein), each CBI process also admits a pathwise construction in terms of a strong solution of a SDE with a Brownian motion and Poisson random measures as driving noise. Namely, let $(\Omega,\cF,(\cF_t)_{t\geq 0},\P)$ be a probability space rich enough to support an $(\F_t)_{t \geq 0}$-Brownian motion, an $(\F_t)_{t \geq 0}$-Poisson random measure $N_{\nu}$ on $(0,\infty)^{2}$ with intensity $ ds \nu(dz)$, and an $(\F_t)_{t \geq 0}$-Possion random measure $N_{\mu}$ on $(0,\infty)^{3}$ with intensity $ds \mu(dz)du$. Then the unique strong solution of
\begin{align}\label{SDE}
 X^x_{t} &= x + \int_0^t (b+\beta X^x_{s})ds+ \sigma \int_0^t\sqrt{X^x_{s}} dB_{s} 
 \\ \notag &\qquad \qquad + \int_0^t\int_{0}^{\infty} z N_{\nu}(ds,dz) + \int_0^t\int_{0}^{\infty}\int_{0}^{X^x_{s}}z \widetilde{N}_{\mu}(ds,dz,du)
\end{align}
with initial condition $x \geq 0$ satisfies \eqref{eq: affine formula} and hence determines a CBI process with parameters $(b,\beta,\sigma,\mu,\nu)$.
Here $\widetilde{N}_{\mu}(dt,dz,du) = N_{\mu}(dt,dz,du) - dt \mu(dz)du$ denotes the corresponding compensated Poisson random measure.

\subsection{Discussion of results and examples}

Let $X^x = (X_t^x)_{t \geq 0}$ be the CBI process with parameters $(b,\beta, \sigma,\mu,\nu)$ constructed from \eqref{SDE}. We call the CBI process \textit{subcritical}, if its drift satisfies $\beta < 0$.
It follows from \cite{pinsky1972} (see also \cite[Corollary 3.21]{li_book}) that $X^x$ has a unique invariant measure $\pi$ and that $X_t^x \Longrightarrow \pi$ weakly as $t \to \infty$ provided that it is subcritical and the big jumps of $\nu$ satisfy the integrability condition $\int_1^{\infty}\log(z)\nu(dz) < \infty$. Moreover, according to \cite{FJR2020b}, this convergence also holds in the Wasserstein-1 distance, provided that $\nu$ satisfies the stronger moment condition $\int_1^{\infty}z \nu(dz) < \infty$. Thus, using \cite{sandric2017}, we conclude that the time averaged process 
\begin{align}\label{eq: time-averaged process}
 Y^x_t := \frac{1}{t}\int_0^t X^x_s ds
\end{align}
converges in $L^1(\Omega, \F,\P)$ to the space-averages, i.e.
\begin{align}\label{eq: LLN L1}
 Y^x_t \longrightarrow \int_0^{\infty}y \pi(dy) = (-\beta)^{-1}\left(b + \int_0^{\infty}z \nu(dz)\right) =: m.
\end{align}
In this work we first strengthen \eqref{eq: LLN L1} to convergence in $L^2$, provided that $\nu,\mu$ have finite second moments for the big jumps. 
\begin{Theorem}[Law of large numbers in $L^2$]\label{thm: lln}
 Let $X^x$ be the subcritical CBI process with parameters $(b,\beta,\sigma, \nu,\mu)$ and $x \geq 0$. Suppose that
 \begin{align}\label{eq: second moments}
  \int_1^{\infty}z^2 \nu(dz) + \int_1^{\infty}z^2 \mu(dz) < \infty.
 \end{align}
 Then the law of large numbers also holds in $L^2$, i.e.
 \[
  \lim_{t \to \infty}\E\left[ \left|\frac{1}{t}\int_0^t X_s^x ds - m \right|^2 \right] = 0.
 \]
\end{Theorem}
A proof of this statement is given in the first part of Section 2. It follows the arguments given in \cite{zhu}, but now adjusted to cover the general class of one-dimensional subcritical CBI processes. Under slightly stronger assumptions on the branching mechanism $R$ one could also obtain a.s. convergence, see \cite{GZ18, LM15, FJKR19}. 

In view of the above, we are now interested in the typical fluctuations of $(Y_t^x)_{t \geq 0}$ around its limit $m$. More precisely, we prove the functional central limit theorem under the second moment condition \eqref{eq: second moments}.
\begin{Theorem}[Functional central limit theorem]\label{thm: clt}
 Let $X^x$ be the subcritical CBI process with parameters $(b,\beta,\sigma, \nu,\mu)$ and $x \geq 0$. Suppose that \eqref{eq: second moments} holds. Then 
 \[
  \sqrt{n}\left(\frac{1}{n}\int_0^{nt} X_s^x ds - mt \right) \Longrightarrow \rho W_t, \qquad t \in [0,1]
 \]
 in law where $W_t$ is a standard Brownian motion and $\rho$ is given by
 \[
  \rho^2 =  \frac{m }{\beta^2}\left( \sigma^2 + \int_0^{\infty}z^2 \mu(dz) \right) + \frac{1}{\beta^2}\int_0^{\infty} z^2 \nu(dz) 
  = \frac{R''(0)m + F''(0)}{\beta^2}.
 \]
\end{Theorem}
The proof of this result is also given in Section 2 where using the stochastic equation \eqref{SDE} we first decompose $\sqrt{n}\left(\frac{1}{n}\int_0^{nt} X_s^x ds - mt\right) = \beta^{-1}n^{-1/2}(X_{nt}^x + x) + \beta^{-1}n^{-1/2}M_{nt}$ for some martingale $(M_t)_{t \geq 0}$. The assertion then follows from the central limit theorem for martingales (see e.g. \cite{jacod_shiryaev}), provided we can show that $[M_{nt}]/n \longrightarrow \beta^2\rho^2 t$ in probability. To prove $[M_{nt}]/n \longrightarrow \beta^2\rho^2 t$, we successively use the law of large numbers for $Y_t^x$ as well as for the stochastic integrals against the Poisson random measures $N_{\nu}, N_{\mu}$. 

Similar results have been studied in \cite{zhu} for the specific choice of L\'evy measures $\nu = c_{\nu} \delta_a$, $\mu = c_{\mu} \delta_a$ with $c_{\nu},c_{\mu} \geq 0$ and $a > 0$. Moreover, in \cite{GZ18} a similar statement was derived for general affine processes under the additional assumption that $\nu,\mu$ are probability measures satisfying a lower bound. Thus, our result complements the existing literature as being applicable to general one-dimensional CBI processes. It is worthwhile to mention that our proofs of Theorem \ref{thm: lln} and Theorem \ref{thm: clt} do not explicitly use the one-dimensional structure of the process and hence can also be extended to general multi-type CBI processes on state-space $\R_+^m$.

In our main result of this work we study the a-typical fluctuations of $Y_t^x$ around the limit $m$. The latter one is described in terms of large deviations for the sequence $(Y^x_t)_{t \geq 0}$ when $t \to \infty$. To formulate this result, let us introduce first
\begin{align*}
 \gamma_R &:= \sup \left\{\gamma\geq 0\colon\int_1^\infty e^{\gamma z}\mu(dz)<\infty \right\},
\\ \gamma_F &:= \sup \left\{\gamma \geq 0 \colon \int_1^\infty e^{\gamma z}\nu(dz)<\infty \right\}.
\end{align*}
Note that the L\'evy measures $\nu,\mu$ have finite exponential moments in the sense that
\begin{align}\label{eq: exponential moment}
 \int_1^{\infty} e^{\gamma z} \left( \nu(dz) + \mu(dz) \right) < \infty \qquad \text{ for some }\gamma>0
\end{align}
if and only if $\gamma_R, \gamma_F > 0$. 
In any case, using dominated convergence it is not difficult to see that $R$ and $F$ are analytic on $(-\infty,\gamma_R)$ and $(-\infty, \gamma_F)$. Moreover, under condition \eqref{eq: second moments} $R$ is strictly convex if and only if
\begin{align}\label{eq: non degeneracy}
 0 < R''(0) = \sigma^2 + \int_0^{\infty} z^2 \mu(dz). 
\end{align}
For such functions $R$ we have the following auxiliary result.
\begin{Lemma}\label{minimum_u_c}
 Let $X^x$ be the subcritical CBI process with parameters $(b, \beta, \sigma, \nu,\mu)$, $x \geq 0$, satisfying \eqref{eq: non degeneracy} and $\gamma_R > 0$. Then $R$ has a unique global minimum $u_c$ which satisfies $u_c \in (0,\gamma_R]$ for $\gamma_R < \infty$ and $u_c \in (0,\infty)$ for $\gamma_R = + \infty$. In any case, letting 
 \[
  \lambda_R := - R(u_c)
 \]
 we find that $\lambda_R \in (0, \infty)$.
\end{Lemma}
A proof of this statement is given in the appendix.
The value $\lambda_R$ acts as a critical threshold for the existence of solutions of the equation 
\begin{equation}\label{equilibria}
 R(u)+\lambda = 0.
\end{equation}
 More precisely, we have the following result. 
\begin{Proposition}\label{prop: y properties}
 Let $X^x$ be the subcritical CBI process with parameters $(b, \beta, \sigma, \nu,\mu)$, $x \geq 0$, satisfying \eqref{eq: non degeneracy} and $\gamma_R > 0$. Then there exists a unique continuous function $y: (-\infty, \lambda_R] \longrightarrow (-\infty,u_c]$ such that
 \[
  R(y(\lambda)) + \lambda = 0, \qquad \forall \lambda \in (-\infty, \lambda_R]
 \]
 with $y(\lambda_R) = u_c$.
 This function is continuously differentiable on $(-\infty, \lambda_R)$ with
 \begin{align}\label{eq: y0 derivative}
  y'(\lambda) = -\frac{1}{R'(y(\lambda))} > 0.
 \end{align}
 Equation \eqref{equilibria} has no solutions for $\lambda > \lambda_R$.
\end{Proposition}
Also here the proof of this proposition is given in the appendix.
Since $y$ is strictly increasing with $y(0) = 0$ and $y(\lambda_R) = u_c > 0$, it has a unique solution $\lambda_F > 0$ of $y(\lambda_F) = \gamma_F$ whenever $0<\gamma_F \leq u_c$. If $0 < u_c < \gamma_F$, then this equation has no solution and we define $\lambda_F = +\infty$. Finally let
\[
 \lambda_c = \min\{\lambda_F, \lambda_R\} > 0.
\]
Then $y(\lambda_c) = \min\{u_c,\gamma_F\} \leq \min\{\gamma_R, \gamma_F\}$. Hence, since $y(\lambda) < y(\lambda_c)$ for $\lambda<\lambda_c$, we find that $R(y(\lambda)), F(y(\lambda))$ are continuously differentiable on $\lambda < \lambda_c$. Moreover, it is not difficult to check that $F$ and $R$ are well-defined on the extended real line when $\lambda = \lambda_c$. We are now prepared to state our main result.

\begin{Theorem}[Large deviation principle]\label{main_result}
 Let $X^x$ be the subcritical CBI process with parameters $(b, \beta, \sigma, \nu,\mu)$, $x \geq 0$, satisfying \eqref{eq: non degeneracy}, $\gamma_R, \gamma_F \in (0,\infty]$, and $F \neq 0$. Then the time-averaged process $(Y^x_t)_{t \geq 0}$ defined by \eqref{eq: time-averaged process} satisfies for each Borel set $A \subset \R_+$
\begin{align}
 - \inf_{y \in A^{\circ} \cap (0, \alpha)} \Lambda^{\ast}(y) 
 &\leq \liminf_{t \to \infty}\frac{1}{t}\log\P\left[ Y^x_t \in A \right] \label{eq: estimate LDP}
 \\ &\leq \limsup_{t \to \infty}\frac{1}{t}\log\P\left[ Y^x_t \in A \right] \leq - \inf_{y \in \overline{A}}\Lambda^{\ast}(y) \notag
\end{align}
where $\Lambda^{\ast}$ is a good rate function given by 
\begin{displaymath}
 \Lambda^\ast(x)=\sup_{\lambda\leq\lambda_c}\left\{\lambda x - F(y(\lambda)) \right\},
\end{displaymath}
and $\alpha \in [m, \infty]$ is determined by
\[
 \alpha = \sup_{ \lambda < \lambda_c}-\frac{F^\prime(y(\lambda))}{R^\prime(y(\lambda))}.
\]
Finally, the rate function $\Lambda^\ast$ satisfies $\Lambda^*(m) = 0$ and $\Lambda^*(x) > 0$ whenever $x \in \R_+ \setminus \{m\}$.
\end{Theorem}
Here $A^{\circ}$ and $\bar{A}$ denote the interior and closure of $A$, respectively. Moreover $\Lambda^{\ast}$ is, by definition, a good rate function if it is level sets $\{x \colon \Lambda^{\ast}(x) \leq r\}$ are compact for any $r$. Under the conditions of Theorem \ref{main_result}, we have $Y_t \longrightarrow \delta_m$ weakly as $t \to \infty$ e.g. due to Theorem \ref{thm: lln}). Consequently, we have $\P[ Y_t \in A] \longrightarrow \delta_m(A)$ as $t \to \infty$ for each Borel set $A \subset [0,\infty)$ with $m \not \in \partial A$. In particular, the good rate function $\Lambda^*$ should be zero at $m$ which is shown in the last part of Theorem \ref{main_result}.

Comparing our main result with the classical large deviation principle, in our case the lower bound contains the intersection $A^{\circ} \cap (0,\alpha)$. Thus, if $\alpha = + \infty$, then it is natural to expect that $(Y^x_t)_{t \geq 0}$ satisfies the large deviation principle with good rate function $\Lambda^{\ast}$. This is precisely the content of the next corollary. For additional details, results, and general background on large deviations we refer to \cite{dembo_zeitouni, varadhan}.

\begin{Corollary}[Large deviation principle]\label{cor: main_result}
 Let $X^x$ be the subcritical CBI process with parameters $(b, \beta, \sigma, \nu,\mu)$, $x \geq 0$, satisfying \eqref{eq: non degeneracy}, $\gamma_R, \gamma_F \in (0,\infty]$, and $F \neq 0$. If 
 \begin{align}\label{eq: boundary condition}
  \lim_{\lambda \nearrow \lambda_c}\frac{F'(y(\lambda))}{|R'(y(\lambda))|} = + \infty,
 \end{align}
 then the time-averaged process $(Y^x_t)_{t \geq 0}$ defined by \eqref{eq: time-averaged process}
 satisfies the large deviation principle with good rate function
 $\Lambda^\ast$, i.e. 
 \begin{align*}
  - \inf_{y \in A^{\circ}} \Lambda^{\ast}(y) 
  &\leq \liminf_{t \to \infty}\frac{1}{t}\log\P\left[ Y^x_t \in A \right] 
  \\ &\leq \limsup_{t \to \infty}\frac{1}{t}\log\P\left[ Y^x_t \in A \right] \leq - \inf_{y \in \overline{A}}\Lambda^{\ast}(y)
 \end{align*}
 holds for each Borel set $A \subset \R_+$.
\end{Corollary}

Our results extend and complement those in \cite{zhu} in various ways.
Indeed, in \cite{zhu} above results were already obtained for the particular class of CBI processes with L\'evy measures given by $\nu = c_{\nu}\delta_a$ and $\mu = c_{\mu}\delta_a$ with constants $c_{\nu},c_{\mu} > 0$, and $a > 0$. In our framework the L\'evy measures $\nu,\mu$ can be arbitrary (including zero) as long as they have finite exponential moments, i.e., $\gamma_F, \gamma_R > 0$ holds. Such a condition is naturally satisfied for the L\'evy measures used in \cite{zhu}. Our extension imposes additional difficulties related with the Laplace transform $\E[e^{\lambda Y_t^x}]$ when $\lambda > 0$. In particular, the generalized Riccati equation $A'(t,\lambda) = R(A(t,\lambda)) + \lambda$, $A(0,\lambda) = 0$ characterising the Laplace transform of $Y_t^x$ may, in general, have only a local solution on $[0, T(\lambda))$ with the explosion time $T(\lambda)$ being closely related with $\gamma_R$ (see Section 3). Furthermore, in contrast to \cite{zhu}, we do not assume Feller's boundary condition which in our case reads as $2b \geq \sigma^2$. For more general conditions and results related with the Feller condition we refer to \cite{MR3264444, MR3263091}. Finally, in \cite{zhu} it is assumed that $\sigma > 0$ which is stronger than our condition \eqref{eq: non degeneracy}. Below we also provide additional comments on the LDP and the assumptions.
\begin{Remark}\label{remark: F zero}
 The condition $F \neq 0$ assures that $m > 0$. If $F = 0$, then $b = \nu = 0$ and hence $m = 0$ and $\alpha=0$. Thus the left bound in \eqref{eq: estimate LDP} becomes $-\infty$ and does not provide any information. For the right bound in \eqref{eq: estimate LDP} we note that $\Lambda^{\ast}(x) = \lambda_c x$ for $x \geq 0$ since $F = 0$, and hence we find for each Borel set $A \subset \R_+$
 \begin{align*}
  \limsup_{t \to \infty}\frac{1}{t}\log\P\left[ Y^x_t \in A \right] \leq - \lambda_c \inf(A).
 \end{align*}
\end{Remark} 
A proof of this statement is given in the last section of this work. Next we comment on the case where $R$ is not strictly convex, i.e., \eqref{eq: non degeneracy} does not hold.

\begin{Remark}
 If \eqref{eq: non degeneracy} does not hold, then $\sigma = \mu = 0$ but our main results still remain true. Indeed, in this case $X^x$ is a pure-jump Ornstein-Uhlenbeck process with $R(u) = \beta u$. Suppose that $\gamma_F > 0$ and $F \neq 0$. Then by inspection of the proofs we find that $y(\lambda) = \lambda |\beta|^{-1}$, $\lambda_c = |\beta|\gamma_F \in (0, \infty]$, and hence
 \[
  \Lambda^*(x) = \sup_{\lambda \leq \lambda_c} \left\{ \lambda x - F\left(\frac{\lambda}{|\beta|}\right) \right\}.
 \]
 In particular, Theorem \ref{main_result} remains true with this choice of $\Lambda^{\ast}$ instead, while Corollary \ref{cor: main_result} requires the additional condition $F'(\gamma_F) = + \infty$ whenever $\gamma_F < \infty$. 
 The case $F = 0$ is excluded here since then $X_t^x = e^{\beta t}x$ is deterministic.
\end{Remark}
Concerning condition \eqref{eq: boundary condition}, the situation is more involved. This condition is certainly satisfied if $\gamma_R = \gamma_F = + \infty$, i.e. $\mu, \nu$ have finite exponential moments of all orders. In all other cases it may happen, depending on $\mu,\nu$, that \eqref{eq: boundary condition} either holds or fails to hold. This is illustrated by the examples below.
\begin{Example}[Pure jump Ornstein-Uhlenbeck process]
 Take $\sigma = 0$, $\mu = 0$, then $R(u) = \beta u$ with $\beta < 0$ and $\gamma_R = + \infty$. Finally, for $\eta \in \R$ let
 \[
  \nu(dz) = \1_{(1,\infty)}(z) \frac{e^{-z}}{z^{1 + \eta}}dz.
 \]
 Then $\gamma_F = 1$, $\lambda_c = |\beta|$, and $y(\lambda) = \lambda |\beta|^{-1}$. Moreover, we have 
 \begin{align*}
  \lim_{\lambda \nearrow \lambda_c}\frac{F'(y(\lambda))}{|R'(y(\lambda))|}
  = \frac{b}{|\beta|} + \frac{1}{|\beta|}\int_1^{\infty}z^{-\eta}dz
 \end{align*}
 and hence condition \eqref{eq: boundary condition} holds if and only if $\eta \leq 1$.
\end{Example}
\begin{Example}[Jump-diffusion CIR process]
 Let $\beta < 0$, $\mu = 0$, and $\sigma > 0$. Then $R(u) = \beta u + \frac{\sigma^2}{2}u^2$ and $\gamma_R = +\infty$. Let $F \neq 0$ with the L\'evy measure $\nu$ satisfying $\gamma_F \in (0,\infty]$. Then $u_c = \frac{|\beta|}{\sigma^2}$, $\lambda_R = \frac{|\beta|^2}{2\sigma^2}$, and
 \[
  y(\lambda) = \frac{|\beta|}{\sigma^2} - \frac{\sqrt{\frac{\beta^2}{\sigma^2} - 2\lambda}}{\sigma}, \qquad \lambda \leq \lambda_R.
 \]
 Moreover, we find that 
 \[
  \lambda_F = \begin{cases}+\infty, & \text{ if } \gamma_F > \frac{|\beta|}{\sigma^2}
  \\ \frac{|\beta|^2}{2\sigma^2} - \frac{\sigma^2}{2}\left(\frac{|\beta|}{\sigma^2} - \gamma_F\right)^2, & \text{ if } 0 < \gamma_F \leq \frac{|\beta|}{\sigma^2}
  \end{cases}
 \]
 and hence
 \[
  \lambda_c = \begin{cases} \frac{|\beta|^2}{2\sigma^2}, & \text{ if } \gamma_F > \frac{|\beta|}{\sigma^2}
  \\ \frac{|\beta|^2}{2\sigma^2} - \frac{\sigma^2}{2}\left(\frac{|\beta|}{\sigma^2} - \gamma_F\right)^2, & \text{ if } 0 < \gamma_F \leq \frac{|\beta|}{\sigma^2}.
  \end{cases}
 \]
 If $\gamma_F \geq \frac{|\beta|}{\sigma^2}$, then $y(\lambda_c) = \frac{|\beta|}{\sigma^2}$ and hence \eqref{eq: boundary condition} holds due to $R'(|\beta|/\sigma^2) = 0$. If $\gamma_F < \frac{|\beta|}{\sigma^2}$, then \eqref{eq: boundary condition} is satisfied whenever $\nu$ satisfies
 \begin{align}\label{eq: 03}
  \int_0^{\infty}e^{\gamma_F z}z \nu(dz) = +\infty.
 \end{align}
 Let us consider two examples of measures $\nu$. Letting
 \[
  \nu(dz) = \1_{(1,\infty)}(z)\frac{e^{-\gamma_F z}}{z^{1 + \eta}}dz
 \]
 with $\gamma_F \in (0,\infty)$ and $\eta \in \R$, we see that \eqref{eq: boundary condition} holds for $\gamma_F \geq |\beta|/\sigma^2$. For $\gamma_F < |\beta|/\sigma^2$ condition \eqref{eq: boundary condition} holds if and only if $\eta \leq 1$.
 
 Letting, for $\rho > 1$ and $\eta < 2$, 
 \[
  \nu(dz) = \1_{(0,\infty)}(z)\frac{e^{-z^{\rho}}}{z^{1+\eta}}dz,
 \]
 we find that $\gamma_F = +\infty$ and hence \eqref{eq: boundary condition} holds.
\end{Example}

\subsection{Structure of the work}

This work is organized as follows. In Section 2 we first prove Theorem \ref{thm: lln}. Afterwards, we establish a law of large numbers for the stochastic integrals against the Poisson random measures $N_{\nu}, N_{\mu}$ appearing in \eqref{SDE}. Finally, we close Section 2 by proving Theorem \ref{thm: clt}. Section 3 is devoted to a detailed asymptotic study of solutions to the generalized Riccati equation $A'(t,\lambda) = R(A(t,\lambda)) + \lambda$. In particular, we study the existence of exponential moments for $Y_t^x$ and consequently express its Laplace transform on $\R$ in terms of $A(t,\lambda)$. Our main results, Theorem \ref{main_result} and Corollary \ref{cor: main_result}, as well as the supplementary Remark \ref{remark: F zero} are subsequently proved in Section 4. The appendix contains some auxiliary results including the proofs of Lemma \ref{minimum_u_c}, Proposition \ref{prop: y properties}, and uniform moment bounds for subcritical CBI processes.

\section{Proof of LLN and CLT}

\subsection{Proof of the law of large numbers}

Here and below we fix the initial condition $x \geq 0$ and write, by abuse of notation, $X_t$ instead of $X_t^x$.
\begin{proof}[Proof of Theorem \ref{thm: lln}]
 Note that  
 \begin{align*}
 \E\left[\left(\frac{1}{t}\int_0^tX_sds - m\right)^2\right]
 &=\frac{1}{t^2}\E\left[\left(\int_0^t X_sds\right)^2\right]
 - \frac{2m}{t}\int_0^t\E[X_s]ds + m^2.
 \end{align*}
 Using the explicit formula for $\E[X_s]$, it is easy to see that $\frac{1}{t}\int_0^t \E[X_s]ds \xrightarrow{t\to\infty} m$. Hence, the assertion is proved once we have shown that 
 \begin{displaymath}
 \frac{1}{t^2}\E\left[\left(\int_0^tX_sds\right)^2\right] \xrightarrow{t\to\infty} m^2.
 \end{displaymath}
 To this end we use the explicit form of the conditional first moment of the process (see Lemma \ref{lemma: moment formula}) to find that
\begin{align*}
 \frac{1}{t^2}\E\left[\left(\int_0^t X_sds\right)^2\right]
 &= \frac{2}{t^2}\int_{\{0<s_1<s_2<t\}}\E[X_{s_1}\E[X_{s_2}\vert X_{s_1}]]ds_1ds_2
 \\ &=\frac{2}{t^2}\int_{\{0<s_1<s_2<t\}}\E[X_{s_1}^2] e^{\beta(s_2 - s_1)}ds_1 ds_2
 \\ & \qquad \qquad + \frac{2m}{t^2}\int_{\{0<s_1<s_2<t\}} \E[X_{s_1}]ds_1ds_2
 \\ & \qquad \qquad - \frac{2m}{t^2}\int_{\{0<s_1<s_2<t\}}\E[X_{s_1}]e^{\beta(s_2 - s_1)}ds_1ds_2. 
\end{align*}
 Using now Proposition \ref{prop: moment bounds} so that $C = \sup_{r \geq 0}\E[X_r^2] < \infty$, we conclude that the first term converges to zero due to
 \begin{align*}
     \frac{2}{t^2}\int_{\{0<s_1<s_2<t\}}\E[X_{s_1}^2]e^{\beta(s_2-s_1)}ds_1ds_2
     &\leq \frac{2C}{t^2}\int_0^t e^{\beta s_2} \int_0^{s_2} e^{-\beta s_1}ds_1 ds_2
     \\ &= \frac{2C}{t^2}\int_0^t \frac{1- e^{\beta s_2}}{-\beta}ds_2
     \\ &\leq \frac{2C}{|\beta|}\frac{1}{t}.
 \end{align*}
 In a very similar way, one can show that also the last term satisfies
 \[
  \lim_{t\to \infty} \frac{2m}{t^2}\int_{\{0<s_1<s_2<t\}}\E[X_{s_1}]e^{\beta(s_2 - s_1)}ds_1ds_2 = 0.
 \]
 Finally, for the second term we obtain
 \begin{align*}
     \frac{2m}{t^2}\int_{\{0<s_1<s_2<t\}} \E[X_{s_1}]ds_1ds_2
     &= \frac{2m}{t^2}\int_0^t \int_0^{s_2} \left( xe^{\beta s_1} + m(1 - e^{\beta s_1}) \right)ds_1 ds_2
     \\ &= \frac{2m x}{t^2}\frac{1}{\beta}\int_0^t  \left( e^{\beta s_2} - 1\right) ds_2
     \\ &\qquad + \frac{2m^2}{t^2}\int_0^t s_2 ds_2 - \frac{2m^2}{t^2}\frac{1}{\beta}\int_0^t \left( e^{\beta s_2} - 1 \right)ds_2
 \end{align*}
 Since the first and third summand here converge to zero, we have
 \[
  \lim_{t \to \infty}\frac{2m}{t^2}\int_{\{0<s_1<s_2<t\}} \E[X_{s_1}]ds_1ds_2
  = \lim_{t \to \infty} \frac{2m^2}{t^2}\int_0^t s_2 ds_2
  = m^2.
 \]
 This proves the assertion.
\end{proof}

We close this subsection with an auxiliary result on the law of large numbers for the Poisson integrals. It will be used in the proof of the central limit theorem.
\begin{Lemma}\label{lemma: lln Poisson}
 Under the conditions of Theorem \ref{thm: clt}, one has in probability
  \begin{align*}
  \frac{1}{t}\int_0^t\int_0^{\infty} z^2 N_{\nu}(ds,dz) \xrightarrow{t\to\infty} \int_0^\infty z^2\nu(dz).
 \end{align*}
 and 
 \begin{displaymath}
  \frac{1}{t}\int_0^t\int_0^{\infty} \int_0^{\infty} z^2 \1_{\{ u \leq X_{s-} \}}  N_{\mu}(ds,dz,du) \xrightarrow{t\to\infty} m\int_0^\infty z^2\mu(dz).
 \end{displaymath}
\end{Lemma}
\begin{proof}
 Fix $\e \in (0,1)$ and $n \in \N$. For the Poisson integral against $N_{\nu}$, we have
\begin{align*}
    &\ \P\left[\left|\frac{1}{t}\int_0^t\int_0^\infty z^2 N_{\nu}(ds,dz)-\int_0^\infty z^2\nu(dz)\right| > \e\right]
    \\ &\leq \P\left[\left|\frac{1}{t}\int_0^t\int_0^\infty z^2 N_{\nu}(ds,dz)- \frac{1}{t}\int_0^t \int_0^n z^2 N_{\nu}(ds,dz)\right| > \e/3\right]
    \\ &\qquad + \P\left[\left|\frac{1}{t}\int_0^t\int_0^n z^2 N_{\nu}(ds,dz)-\int_0^n z^2\nu(dz)\right| > \e/3\right]
    \\ &\qquad + \P\left[\left| \int_0^n z^2 \nu(dz)-\int_0^\infty z^2\nu(dz)\right| > \e/3\right]
    \\ &\leq  \frac{6}{\e}\int_n^{\infty} z^2 \nu(dz)
    + \frac{9}{\e^2}\E\left[\left(\frac{1}{t}\int_0^t\int_0^n z^2 N_{\nu}(ds,dz)-\int_0^n z^2\nu(dz)\right)^2\right].
\end{align*}
 The last term can be estimated by
 \begin{align*}
    \E\left[\left(\frac{1}{t}\int_0^t\int_0^n z^2 N_{\nu}(ds,dz)-\int_0^n z^2\nu(dz)\right)^2\right] 
    &= \frac{1}{t^2}\E\left[\left(\int_0^t\int_0^n z^2\widetilde{N}_\nu(ds,dz)\right)^2\right]
    \\ &= \frac{1}{t^2}\E\left[\int_0^t\int_0^n z^4 N_\nu(ds,dz)\right]
    \\ &=\frac{1}{t}\int_0^n z^4\nu(dz).
 \end{align*}
 Thus we arrive at
 \begin{align*}
     &\ \P\left[\left|\frac{1}{t}\int_0^t\int_0^\infty z^2 N_{\nu}(ds,dz)-\int_0^\infty z^2\nu(dz)\right| > \e\right]
     \\ &\leq \frac{6}{\e}\int_n^{\infty} z^2 \nu(dz) + \frac{9}{\e^2 t}\int_0^n z^4 \nu(dz).
 \end{align*}
 Letting first $t \to \infty$ and then $n \to \infty$ proves the assertion for the integral against $N_{\nu}$.
 Similar arguments yield for the second Poisson integral
\begin{align*}
    &\ \P\left[ \left|\frac{1}{t}\int_0^t\int_0^{\infty} \int_0^{\infty} z^2 \1_{\{ u \leq X_{s-} \}}  N_{\mu}(ds,dz,du) - \frac{1}{t}\int_0^t\int_0^{\infty} z^2 X_{s-} ds\mu(dz) \right| > \e \right]
    \\ &\leq \P\left[ \left| \frac{1}{t}\int_0^t\int_n^{\infty} \int_0^{\infty} z^2 \1_{\{ u \leq X_{s-} \}}  N_{\mu}(ds,dz,du) \right| > \e/3 \right]
    \\ &\ + \P\left[ \left| \frac{1}{t}\int_0^t\int_0^{n} \int_0^{\infty} z^2 \1_{\{ u \leq X_{s-} \}}  N_{\mu}(ds,dz,du) - \frac{1}{t}\int_0^t\int_0^{n} z^2 X_{s-} ds\mu(dz) \right| > \e/3 \right]
    \\ &\ + \P\left[ \left| \frac{1}{t}\int_0^t\int_n^{\infty} z^2 X_{s-} ds\mu(dz) \right| > \e/3 \right]
    \\ &\leq \frac{6}{\e t}\int_0^t \int_n^{\infty} z^2 \E[X_s] \mu(dz)ds
    \\ &\ + \frac{9}{\e^2}\E\left[ \left( \frac{1}{t}\int_0^t\int_0^{n} \int_0^{\infty} z^2 \1_{\{ u \leq X_{s-} \}}  N_{\mu}(ds,dz,du) - \frac{1}{t}\int_0^t\int_0^{n} z^2 X_{s-} ds\mu(dz) \right)^2 \right].
 \end{align*}
 The last term can be estimated by
 \begin{align*}
    &\ \E\left[ \left( \frac{1}{t}\int_0^t\int_0^{n} \int_0^{\infty} z^2 \1_{\{ u \leq X_{s-} \}}  N_{\mu}(ds,dz,du) - \frac{1}{t}\int_0^t\int_0^{n} z^2 X_{s-} ds\mu(dz) \right)^2 \right]
    \\ &=  \frac{1}{t^2 }\E\left[ \left( \int_0^t\int_0^{n} \int_0^{\infty} z^2 \1_{\{ u \leq X_{s-} \}}  \widetilde{N}_{\mu}(ds,dz,du) \right)^2 \right]
    \\ &= \frac{1}{t^2}\E\left[  \int_0^t\int_0^{n} \int_0^{\infty} z^4 \1_{\{ u \leq X_{s-} \}}  N_{\mu}(ds,dz,du)  \right]
    \\ &= \frac{1}{t} \int_0^{n} z^4 \mu(dz) \left(\frac{1}{t}\int_0^t\E[X_{s-}] ds \right)
 \end{align*}
 which tends to zero as $t \to \infty$.
 Again, letting first $t \to \infty$ and then $n \to \infty$ we obtain 
 \[
  \P\left[ \left|\frac{1}{t}\int_0^t\int_0^{\infty} \int_0^{\infty} z^2 \1_{\{ u \leq X_{s-} \}}  N_{\mu}(ds,dz,du) - \frac{1}{t}\int_0^t\int_0^{\infty} z^2 X_{s-} ds\mu(dz) \right| > \e \right] \to 0
 \]
 as $t \to \infty$.
 The assertion now follows from Theorem \ref{thm: lln}, i.e.
 \begin{align*}
    \E\left[ \left( \frac{1}{t}\int_0^t\int_0^{\infty} z^2 X_{s-} ds\mu(dz) - m \int_0^{\infty} z^2 \mu(dz) \right)^2 \right] \longrightarrow 0.
 \end{align*}
\end{proof}
This proof shows that the convergence can be strenghtened to $L^2$ provided that $\mu,\nu$ satisfy the stronger moment condition 
\[
 \int_1^{\infty}z^4 \left( \mu(dz) + \nu(dz) \right) < \infty.
\]
In such a case no approximation by cutting the big jumps (as done in the above proof) is necessary.

\subsection{Proof of the central limit theorem}

As before, we fix $x \geq 0$ and write $X_t$ instead of $X_t^x$.
\begin{proof}[Proof of Theorem \ref{thm: clt}]
 Using \eqref{SDE} we obtain
 \begin{align*}
     M_t &:= X_t - x - \int_0^t \left( \widehat{b} + \beta X_s \right)ds
     \\ &= \sigma \int_0^t \sqrt{X_s}dB_s + \int_0^t \int_0^{\infty}\int_0^{\infty} z \1_{\{u \leq X_{s-}\}} \widetilde{N}_{\mu}(ds,dz,du) 
     \\ &\qquad + \int_0^t \int_0^{\infty} z \widetilde{N}_{\nu}(ds,dz),
 \end{align*}
 where $\widehat{b}:=b+\int_{0}^{\infty} z \nu(dz)$,
 and $(M_t)_{t \geq 0}$ is a martingale. Dividing by $\beta$ and rearranging terms we arrive at $\int_0^t X_s ds - mt = \beta^{-1}X_t - \beta^{-1}x - \beta^{-1}M_t$ and hence
 \[
  \frac{1}{n}\int_0^{nt} X_s ds - mt = \beta^{-1}\frac{X_{nt}}{n} - \beta^{-1}\frac{x}{n} - \beta^{-1}\frac{M_{nt}}{n}
 \]
 Thus multiplying by $\sqrt{n}$ we obtain
 \begin{align}\label{eq: 00}
     \sqrt{n} \left( \frac{1}{n}\int_0^{nt} X_s ds - mt \right) &= \beta^{-1} \frac{X_{nt}}{\sqrt{n}} - \frac{\beta^{-1}x}{\sqrt{n}} - \beta^{-1} \frac{M_{nt}}{\sqrt{n}}.
 \end{align}
 The quadratic variation of the last term satisfies
 $\left[ M_{nt}/\sqrt{n} \right] = [M_{nt}]/n$ with
 \begin{align*}
  [M_{nt}] &= \sigma^2 \int_0^{nt} X_s ds + \int_0^{nt} \int_0^{\infty}\int_0^{\infty} z^2 \1_{\{u \leq X_{s-}\}} N_{\mu}(ds,dz,du) 
  \\ &\qquad \qquad + \int_0^{nt} \int_0^{\infty} z^2 N_{\nu}(ds,dz).
 \end{align*}
 Thus, using the law of large numbers for $X$ (see Theorem \ref{thm: lln}) and the Poisson random measures (see Lemma \ref{lemma: lln Poisson}), we conclude that
 $\left[ -\beta^{-1} M_{nt} / \sqrt{n} \right] = |\beta|^{-2} [M_{nt}] /n \longrightarrow \rho^2 t$. Thus using the central limit theorem for square-integrable martingale (see \cite[Cor. VIII-3.24]{jacod_shiryaev}), we conclude that $(-\beta)^{-1}M_{nt}/\sqrt{n} \Longrightarrow \rho W_t$ in law where $W_t$ is a standard Brownian motion.
 Finally, since $\E[X_{nt}]$ is uniformly bounded, we have $\sqrt{n}^{-1}\E[X_{nt}] \to 0$ and hence $X_{nt}/\sqrt{n} \to 0$ in probability. Thus the first two terms in \eqref{eq: 00} converge to zero in probability. This proves the assertion.
\end{proof}

\section{Moment generating function}

\subsection{Asymptotic analysis of generalized Riccati equation}

In this section we suppose that $(X_t^x)_{t \geq 0}$ is a subcritical CBI process with parameters $(b,\beta, \sigma,\nu,\mu)$ satisfying $\gamma_R > 0$.
Below we study the existence, uniqueness and long-time behavior of the solution of the equation
\begin{align}\label{eq: A equation}
 A'(t) &= R(A(t)) + \lambda, \qquad A(0) = 0
\end{align}
where $\lambda \in \R$. We start with the simpler case $\lambda \leq \lambda_R$.
\begin{Lemma}\label{lemma: riccati global existence}
 For each $\lambda \in (- \infty,\lambda_R]$, Equation \eqref{eq: A equation} has a unique global solution with
 \[
  \begin{cases}
   y(\lambda) < A(t,\lambda) \leq 0, & \text{ for } \lambda < 0
   \\ A(t,\lambda) = 0, & \text{ for } \lambda = 0
   \\ 0 \leq A(t,\lambda) < y(\lambda), & \text{ for } \lambda \in (0, \lambda_R]
  \end{cases}
 \]
 satisfying
 \begin{align}\label{eq: A convergence to y0}
  \lim_{t \to \infty}A(t,\lambda) = y(\lambda).
 \end{align} 
\end{Lemma}
\begin{proof}
 Since $R$ is twice continuously differentiable, it is locally Lipschitz continuous and hence this equation has at most one solution. For $\lambda=0$, this solution is explicitly given by $A(t)=0 = y(0)$ for all $t\geq 0$ and therefore the assertion is proved. 
 
 Hence, let $\lambda\in(-\infty,\lambda_R]\setminus\{0\}$. The existence of a solution can be obtained from separation of variables. Indeed, any solution is determined by the relation
 \[
  H(A(t)) = \int_0^{A(t)}\frac{du}{R(u) + \lambda} = t.
 \]
 where we have set $H(x) = \int_0^x \frac{du}{R(u) + \lambda}$. Below we prove that $H$ is invertible so that the unique solution is given by $A(t) = H^{-1}(t)$ for all $t \geq 0$.
 
 \textit{Case 1:} $\lambda < 0$. In this case we have $R(u) + \lambda < 0$ for $u \in (y(\lambda),0]$ and hence $H(x) = - \int_x^0 \frac{du}{R(u) + \lambda}$ is well-defined and decreasing on $(y(\lambda),0]$. Hence it has an
 inverse $H^{-1}: [0,T^*) \longrightarrow (y(\lambda),0]$ with
  \[
   T^* = \int_0^{y(\lambda)} \frac{du}{R(u) + \lambda} \in (0,\infty].
  \]
  To prove that this solution is also global, we need to show that $T^* = + \infty$. To this end, we note that $R(y(\lambda)) + \lambda = 0$ but $R'(y(\lambda)) \neq 0$ since $y(\lambda) \neq u_c$ for $\lambda < \lambda_R$. Hence $y(\lambda)$ is a simple zero and we can write $R(u) + \lambda = (u - y(\lambda))\widetilde{R}(u)$ for some $C^1$-function $\widetilde{R}$ in the neighborhood of $y(\lambda)$. Thus $T^* = +\infty$. Now suppose that \eqref{eq: A convergence to y0} does not hold. Then there exists $\e > 0$ and a sequence $(t_n)_{n \in \N}$ satisfying $t_n \nearrow \infty$ and $y(\lambda) + \e < A(t_n,\lambda) \leq 0$. Hence we obtain
  \[
   t_n = \int_0^{A(t_n,\lambda)}\frac{du}{R(u) + \lambda} \leq \int_0^{y(\lambda)+\e} \frac{du}{R(u) + \lambda} < \infty
  \]
  where we have used that  the integrand is nonpositive. Letting $n \to \infty$ yields a contradiction. Thus \eqref{eq: A convergence to y0} holds.
     
  \textit{Case 2:} $\lambda \in (0,\lambda_R]$. In this case we have $R(u) + \lambda > 0$ for $u \in [0, y(\lambda))$ and hence $H$ is well-defined on $[0,y(\lambda))$ and increasing. Its inverse $H^{-1}: [0,T^*) \longrightarrow [0,y(\lambda))$ also satisfies $T^*= + \infty$ by exactly the same argument (with a simple zero when $\lambda < \lambda_R$ and possibly a zero of order 2 when $\lambda = \lambda_R$). Suppose that \eqref{eq: A convergence to y0} does not hold. Then there exists $\e > 0$ and a sequence $(t_n)_{n \in \N}$ satisfying $t_n \nearrow \infty$ and $A(t_n,\lambda) < y(\lambda) - \e$. Hence we obtain
  \[
   t_n = \int_0^{A(t_n,\lambda)}\frac{du}{R(u) + \lambda} \leq \int_0^{y(\lambda)-\e} \frac{du}{R(u) + \lambda} < \infty.
  \]
  which gives a contradiction for $n$ large enough. Thus \eqref{eq: A convergence to y0} holds.
\end{proof}

\begin{Lemma}\label{lemma: riccati solution explosion}
 For each $\lambda > \lambda_R$, Equation \eqref{eq: A equation} has a unique solution up to time 
 \[
  T(\lambda) =  \int_0^{\gamma_R} \frac{du}{R(u) + \lambda} \in (0, \infty].
 \]
 This solution satisfies $0 \leq A(t,\lambda) < \gamma_R$ for $t \in [0, T(\lambda))$ and
 \begin{align}\label{eq: A convergence to gammaR}
  \lim_{t \nearrow T(\lambda)}A(t,\lambda) = \gamma_R.
 \end{align}
 Finally, $\gamma_R < \infty$ implies that $T(\lambda) < \infty$.
\end{Lemma}
\begin{proof}
 Note that $H(x) = \int_0^x \frac{du}{R(u) + \lambda}$ is well-defined and increasing for $x \in [0, \gamma_R)$ (since the integrand is nonnegative). Hence it has an inverse $H^{-1}: [0,T(\lambda)) \longrightarrow [0,\gamma_R)$ which shows the existence of a solution of \eqref{eq: A equation} with $0 \leq A(t,\lambda) < \gamma_R$. Uniqueness follows from the local Lipschitz continuity of $R$.
 
 For \eqref{eq: A convergence to gammaR}, first consider the case $\gamma_R=\infty$. Assume first that $T(\lambda)=\infty$ holds. Since $\lambda>\lambda_R$, we have $R(u)+\lambda\geq\lambda-\lambda_R>0$ for all $u\in\R$, and hence as $t \to \infty$
 \begin{displaymath}
  A(t,\lambda)=\int_0^t R(A(s,\lambda))+\lambda ds\geq(\lambda-\lambda_R)t\longrightarrow +\infty.
 \end{displaymath}
 Next, assume that $\gamma_R=\infty$ but $T(\lambda)<\infty$. Suppose that \eqref{eq: A convergence to gammaR} does not hold, i.e. there exists $0\leq K< \infty$ and a sequence $\{t_n\}_{n\in\N}$ with $t_n\nearrow T(\lambda)$ such that $A(t_n,\lambda)\leq K$ for all $n\in\N$. Hence,
 \begin{displaymath}
  t_n=\int_0^{A(t_n,\lambda)}\frac{du}{R(u)+\lambda}\leq\int_0^K\frac{du}{R(u)+\lambda}=T(\lambda)-\int_K^\infty\frac{du}{R(u)+\lambda}.
 \end{displaymath}
 Letting $n\to\infty$ gives $\int_K^\infty\frac{du}{R(u)+\lambda}=0$ which is impossible since the integrand is continuous and strictly positive. Thus \eqref{eq: A convergence to gammaR} holds.
 
 Now consider the case $\gamma_R\in(0,\infty)$. Since $R(u) + \lambda \geq \lambda - \lambda_R$, we see that $T(\lambda) \leq \frac{\gamma_R}{\lambda - \lambda_R} < \infty$. Suppose now that \eqref{eq: A convergence to gammaR} does not hold. Since $A(t,\lambda) < \gamma_R$, we find $\e > 0$ and a sequence $(t_n)_{n \in \N}$ such that $t_n \nearrow T(\lambda)$ and $A(t_n,\lambda) < \gamma_R - \e$. Hence
 \begin{align*}
  t_n = \int_0^{A(t_n,\lambda)}\frac{du}{R(u) + \lambda}
  < \int_0^{\gamma_R - \e} \frac{du}{R(u) + \lambda}
  = T(\lambda) - \int_{\gamma_R - \e}^{\gamma_R} \frac{du}{R(u) + \lambda}.
 \end{align*}
 Letting $n \to \infty$ gives $\int_{\gamma_R - \e}^{\gamma_R} \frac{du}{R(u) + \lambda} = 0$ which is impossible since the integrand is continuous and strictly positive. Thus \eqref{eq: A convergence to gammaR} holds.
\end{proof}

A graphic representation of the behaviour of $A^\prime(t,\lambda)=R(A(t,\lambda))+\lambda$ with respect to the different cases of $\lambda$ in the case of $R(u)=\beta u+\frac{\sigma^{2}}{2}u^{2}$ can be found in Figure \ref{graphics_r} in the appendix. This especially illustrates the transition at the critical value $\lambda_R$.

\subsection{Computing the moment generating function}
In this section we suppose that $(X_t^x)_{t \geq 0}$ is a subcritical CBI process with parameters $(b,\beta, \sigma,\nu,\mu)$ satisfying $\gamma_R > 0$ and additionally $\gamma_F > 0$. For fixed $t\in[0,\infty)$, we define the log-moment generating function of the integrated CBI process $Y_t^x$ by
\[
 \Lambda_t(\lambda) = \log\left( \E\left[ e^{\lambda Y^x_t} \right] \right), \qquad \lambda \in \R.
\]
Note that $\Lambda_t(\lambda) \in (-\infty, 0]$ whenever $\lambda \leq 0$. However, if $\E\left[ e^{\lambda Y_t^x} \right] = + \infty$ for some $\lambda > 0$, then also $\Lambda_t(\lambda) = + \infty$. Based on the next result, we characterize $\Lambda_t$ in terms of solutions to a generalized Riccati equation. 
\begin{Proposition}\label{integrated_lt}
 The following assertions hold:
 \begin{enumerate}
  \item[(a)] For each $\lambda \in (-\infty, \lambda_R]$, the affine transformation formula holds for $t \geq 0$:
  \begin{equation}\label{eq: affine trafo}
   \E\left[e^{\lambda \int_0^t X_s^x ds} \right]=\exp \left(x A(t,\lambda) + \int_0^t F(A(s,\lambda))ds \right).
  \end{equation}
  Here $A$ is the unique solution of \eqref{eq: A equation} and the left-hand side is finite if and only if the right-hand side is finite.
  \item[(b)] For $\lambda > \lambda_R$ \eqref{eq: affine trafo} holds for $t \in [0,T(\lambda)]$ with the left side being finite if and only if the right side is finite. For $t > T(\lambda)$ we have
  \begin{align}\label{eq: infinite exponential moment}
   \E\left[e^{\lambda \int_0^t X_s^xds}\right] = + \infty.
  \end{align}
 \end{enumerate} 
\end{Proposition}
\begin{proof}
 (a) For $\lambda \leq 0$ the assertion is an immediate consequence of \cite[Chapter 11]{dfs}. In this case both sides are clearly finite. For $\lambda \in (0,\lambda_R]$ we note that \eqref{eq: A equation} has a unique global solution $A(t,\lambda) \in [0, y(\lambda))$ for $t>0$ due to Lemma \ref{lemma: riccati global existence} and hence the assertion follows from \cite[Theorem 2.14 and Proposition 3.3]{keller-ressel_mayerhofer} (for convenience of the reader a formulation of their result is given in the appendix).
 
 (b) Suppose that $\lambda > \lambda_R$. Then the affine transformation formula holds on $[0,T(\lambda)]$ due to Lemma \ref{lemma: riccati solution explosion} combined with \cite[Theorem 2.14 and Proposition 3.3]{keller-ressel_mayerhofer}. For $t > T(\lambda)$ we obtain \eqref{eq: infinite exponential moment} from \cite[Proposition 3.3]{keller-ressel_mayerhofer} and the fact that \eqref{eq: A equation} has no solution for $t > T(\lambda)$. This proves the assertion.
\end{proof}

Next, we study the limit of the scaled moment generating function 
\[
 \Lambda(\lambda) = \lim_{t \rightarrow \infty} \frac{1}{t} \Lambda_{t}(t\lambda) 
 = \lim_{t \to \infty} \frac{1}{t}\log\left(\E\left( e^{\lambda \int_0^t X_s^x ds } \right) \right)
\]
for each $\lambda \in \R$. More precisely, we prove the following proposition.
\begin{Proposition}\label{limit_mgf}
 For each $\lambda \in \R$, we have 
 \begin{equation*}
 \Lambda(\lambda)=\begin{cases}
  F(y(\lambda)), & \text{ if } \lambda \in (-\infty, \lambda_c]
  \\ +\infty, & \text{ if } \lambda >\lambda_{c}.
\end{cases}
\end{equation*}
where $F(y(\lambda_c)) > - \infty$.
\end{Proposition} 
\begin{proof}
 For $\lambda > \lambda_c$ and $F \neq 0$, we first assume that $\lambda_R \leq \lambda_F$ so that $\lambda_c = \lambda_R$. If $T(\lambda)<\infty$,
 then Proposition \ref{integrated_lt}.(b) yields for $t > T(\lambda)$ that $\Lambda_t(t\lambda) = \infty$ and hence $\Lambda(\lambda) = \infty$. If $T(\lambda)=+\infty$, then Lemma \ref{lemma: riccati solution explosion} implies that $\gamma_R=+\infty$ and $\lim_{t\to\infty}A(t,\lambda)=+\infty$. If $b=0$, let $a> 0$ such that $\nu([a,\infty))>0$ (otherwise, $a$ can be chosen arbitrarily). Since $\lambda>\lambda_c$, $A$ is increasing and hence,
 \begin{align*}
  \frac{1}{t}\int_0^tF(A(s,\lambda))ds
  &=\frac{1}{t}\int_0^{\frac{t}{2}}F(A(s,\lambda))ds+\frac{1}{t} \int_{\frac{t}{2}}^tF(A(s,\lambda))ds
  \\
  &\geq\frac{1}{t}\int_{\frac{t}{2}}^tF(A(t/2,\lambda))ds
  \\
  &=\frac{1}{2}F(A(t/2,\lambda))
  \\
  &\geq \frac{b}{2}A(t/2,\lambda)+\frac{1}{2}\nu([a,\infty))(e^{aA(t/2,\lambda)}-1) \xrightarrow{t\to\infty} +\infty.
 \end{align*}
 If $\lambda_F<\lambda_R$, then $\lambda > \lambda_c = \lambda_F$. For $\lambda > \lambda_R > \lambda_F$ we can argue as above (since $\Lambda_t(t\lambda) = +\infty$ for $t > T(\lambda)$). Assume that $\lambda \in (\lambda_F, \lambda_R]$. Then noting that $y(\lambda_F)=\gamma_F$, that $\lambda\mapsto y(\lambda)$ is strictly increasing, and using Lemma \ref{lemma: riccati global existence} so that $A(t,\lambda) \nearrow y(\lambda) > \gamma_F$, we find $t_0 \geq 0$ such that $A(t,\lambda)>\gamma_F$ holds for all $t\geq t_0$. Hence, $F(A(t,\lambda))=+\infty$ for all $t\geq t_0$. Therefore, the claim is also true by Proposition \ref{integrated_lt}.(a).
 
 For $\lambda > \lambda_c$ and $F = 0$, we consider first the case where $T(\lambda)<\infty$. Then $\Lambda_t(t \lambda) = +\infty$ and hence $\Lambda(\lambda) = + \infty$ whenever $t > T(\lambda)$. For $T(\lambda) = + \infty$ we have $A(t,\lambda) \nearrow +\infty$ and hence we obtain from l'Hospital $\lim_{t\to\infty}\frac{A(t,\lambda)}{t}=\lim_{t\to\infty}R(A(t,\lambda))+\lambda=\infty$.
 
 So assume that $\lambda \leq \lambda_c$. In this case we find by Proposition \ref{integrated_lt}.(a)
 \begin{align}\label{eq: 02}
  \Lambda(\lambda) = \lim_{t \rightarrow \infty} \frac{1}{t}  \left( x A(t,\lambda)+ \int_0^t F(A(s,\lambda))ds\right),
 \end{align}
 where $A$ is given by \eqref{eq: A equation} and $\int_0^t F(A(s,\lambda))ds$ is finite for each $t \geq 0$. Noting that $\lim_{t \to \infty}A(t,\lambda) = y(\lambda)$, we obtain 
 \begin{align}\label{eq: 01}
  \lim_{t \to \infty}\frac{A(t,\lambda)}{t} = 0.
 \end{align}
 If $F(y(\lambda)) < \infty$, then it not difficult to see that $\lim_{t \to \infty}\frac{1}{t}\int_0^t F(A(s,\lambda))ds = F(y(\lambda))$, and hence 
 \[
  \lim_{t \rightarrow \infty} \frac{1}{t}  \left( x A(t,\lambda)+ \int_0^t F(A(s,\lambda))ds\right)
  = F(y(\lambda)) \in \R.
 \]
 Suppose now that $F(y(\lambda)) = + \infty$. Then we necessarily have $\lambda = \lambda_c = \lambda_F$ and $\int_1^{\infty}e^{\gamma_F z}\nu(dz) = \infty$. Take $\e > 0$ arbitrary. Since $A(s,\lambda) \to y(\lambda) = \gamma_F$, we find some $T > 0$ such that 
 $A(s,\lambda) > y(\lambda) - \e$ for $s \geq T$. Hence we obtain for $t \geq T$
 \[
  \frac{1}{t}\int_0^t F(A(s,\lambda))ds 
  \geq \frac{1}{t}\int_{T}^{t}F(y(\lambda) - \e)ds
  \geq \frac{t - T}{t}\int_1^{\infty}e^{(y(\lambda)-\e)z}\nu(dz).
 \]
 Letting $t \to \infty$ shows that 
 \[
  \liminf_{t \to \infty}\frac{1}{t}\int_0^t F(A(s,\lambda))ds \geq \int_1^{\infty}e^{(y(\lambda)-\e)z}\nu(dz), \qquad \forall \e > 0.
 \]
 Letting $\e \searrow 0$ proves that $\frac{1}{t} \int_0^t F(A(s,\lambda))ds \to \infty$. This combined with \eqref{eq: 02} and \eqref{eq: 01} proves that $\Lambda(\lambda) = F(y(\lambda))$.
\end{proof}

\section{Proofs of the main results}

\subsection{Proof of Theorem \ref{main_result}}

\begin{proof}[Proof of Theorem \ref{main_result}]
 It follows from Proposition \ref{limit_mgf} that
 $\Lambda(\lambda)$ is exists in $(-\infty,\infty]$ for each $\lambda \in \R$. Moreover, we have $\Lambda(\lambda) = F(y(\lambda)) < \infty$ whenever $\lambda \in (-\infty, \lambda_c) =: \mathcal{D}_{\Lambda}^{\circ}$. According to \cite[Lemma 2.3.9]{dembo_zeitouni}, $\Lambda$ is convex and $\Lambda^{\ast}$ is a good rate function. Applying the G\"artner-Ellis Theorem (see \cite[Theorem 2.3.6.(a) and (b)]{dembo_zeitouni}) yields 
 \begin{align*}
  - \inf_{y \in A^{\circ} \cap \mathcal{E}} \Lambda^{\ast}(y) 
  &\leq \liminf_{t \to \infty}\frac{1}{t}\log\P\left[ Y^x_t \in A \right] 
  \\ &\leq \limsup_{t \to \infty}\frac{1}{t}\log\P\left[ Y^x_t \in A \right] \leq - \inf_{y \in \overline{A}}\Lambda^{\ast}(y)
 \end{align*}
 where $\mathcal{E}$ denotes the set of exposed points of $\Lambda^{\ast}$ whose exposing hyperplane belongs to $(-\infty, \lambda_c)$ (see \cite[Definition 2.3.3]{dembo_zeitouni}). Note that 
 \[
  \alpha \geq - \frac{F'(y(0))}{R'(y(0))} = m > 0.
 \]
 Below we prove \eqref{eq: estimate LDP} by showing that
 \[
  - \inf_{y \in A^{\circ} \cap (0,\alpha)}\Lambda^{\ast}(y) \leq - \inf_{y \in A^{\circ} \cap \mathcal{E}}\Lambda^{\ast}(y).
 \]
 For this purpose it suffices to show that $(0,\alpha) \subset \mathcal{E}$. So let $y \in (0,\alpha)$. Since by Proposition \ref{prop: y properties}, $\eta \longmapsto \Lambda^\prime(\eta)=-\frac{F^\prime(y(\eta))}{R^\prime(y(\eta))}>0$ is continuous on $(-\infty,\lambda_c)$, and we have $\lim_{\eta\to-\infty}\Lambda^\prime(\eta)=0$, we find some $\widetilde{\eta} \in (-\infty, \lambda_c)$ such that $y=\Lambda^\prime(\widetilde{\eta})$. Applying \cite[Lemma 2.3.9]{dembo_zeitouni} yields $y \in \mathcal{E}$. Thus we have shown that $(0,\alpha) \subset \mathcal{E}$.
 
 Finally, we prove that $\Lambda^{\ast}(m) = 0$, and $\Lambda^{\ast}(x) > 0$ for $x \geq 0$ with $x \neq m$. For fixed $x\in\R$, set $\varphi_x(\lambda)=\lambda x-F(y(\lambda))$ and recall that $\lambda\mapsto y(\lambda)$ is strictly increasing and continuous. Hence it has an inverse $\lambda: (-\infty, y_c) \longrightarrow (-\infty,\lambda_c)$ where $y_c = \min\{u_c,\gamma_F\}$. Define the function $\psi_x(y)=\varphi_x(\lambda(y))=\lambda(y)x-F(y)$.
 Hence, we have 
 \[
  \Lambda^{\ast}(x) = \sup_{y \leq y_c} \psi_x(y).
 \]
 Since $\lambda'(y) = \frac{1}{y^\prime(\lambda(y))}$, the derivatives of $\psi_x$ are given by
 $\psi_x^\prime(y) = -R^\prime(y)x-F^\prime(y)$ and $\psi_x^{\prime\prime}(y) = -R^{\prime\prime}(y)x-F^{\prime\prime}(y)$.
 Assume first that $\nu\neq 0$. We have for each $x \geq 0$
 \begin{align*}
  \psi_x''(y) &= - \left( \sigma^2 + \int_0^{\infty}e^{y z}z^2 \mu(dz)\right)x - \int_{0}^{\infty}z^2 e^{y z}\nu(dz) 
  \\ &\leq - \int_{0}^{\infty}z^2 e^{y z}\nu(dz)  < 0.
 \end{align*}
 Thus $\psi_x$ is strictly concave for each $x \geq 0$. Since $\lambda(0) = 0$, we find that
 \begin{equation}\label{global_max}
  \psi_m^\prime(0) = -R^\prime(0)m-F^\prime(0)=-\beta m-\Big(b+\int_0^\infty z\nu(dz)\Big) = 0.
 \end{equation}
 Thus $y = 0$ is a global maximum for $\psi_m$ which proves that $\Lambda^{\ast}(m) = \psi_m(0) = 0$. On the other hand, if $x \geq 0$ is such that $x\neq m$, we have 
 $\Lambda^{\ast}(x) \geq \psi_x(0) = 0$. 
 By replacing $m$ with $x$ in \eqref{global_max}, it is easy to see that $\psi^\prime_x(0)\neq 0$. Therefore, necessarily $\psi_x(y) > 0$ holds for $y < 0$ or $y > 0$ close enough to zero. For such values of $y$ we obtain $\Lambda^{\ast}(x) \geq \psi_x(y) > 0$ which proves the assertion.
 
 Finally, consider the case $\nu=0$. Since we assumed $F\neq 0$, we necessarily have $b>0$. For $x>0$, we have
 \begin{displaymath}
  \psi_x^{\prime\prime}(y)=- \left( \sigma^2 + \int_0^{\infty}e^{y z}z^2 \mu(dz)\right)x - \int_{0}^{\infty}z^2 e^{y z}\nu(dz)   < 0
 \end{displaymath}
 by \eqref{eq: non degeneracy} and we may proceed as in the case $\nu\neq 0$. For $x=0$, note that
 \begin{displaymath}
  \Lambda^\ast(0)=\sup_{y\leq y_c}\psi_0(y)=\sup_{y\leq y_c}-by=+\infty.
 \end{displaymath}
 Hence, the assertion is also true in this case.
\end{proof}

\subsection{Proof of Remark \ref{remark: F zero}}

\begin{proof}[Proof of Remark \ref{remark: F zero}]
 If $F = 0$, then it follows from Proposition \ref{limit_mgf} that
 \[
  \Lambda(\lambda) = \begin{cases}0, & \lambda \leq \lambda_c \\ +\infty, & \lambda > \lambda_c.
  \end{cases}
 \]
 Applying the G\"artner-Ellis Theorem (see \cite[Theorem 2.3.6.(a) and (b)]{dembo_zeitouni}) yields 
 \begin{align*}
  - \inf_{y \in A^{\circ} \cap \mathcal{E}} \Lambda^{\ast}(y) 
  &\leq \liminf_{t \to \infty}\frac{1}{t}\log\P\left[ Y^x_t \in A \right] 
  \\ &\leq \limsup_{t \to \infty}\frac{1}{t}\log\P\left[ Y^x_t \in A \right] \leq - \inf_{y \in \overline{A}}\Lambda^{\ast}(y)
 \end{align*}
 where $\mathcal{E}$ denotes the set of exposed points of $\Lambda^{\ast}$ whose exposing hyperplane belongs to $(-\infty, \lambda_c)$. It is easy to see that $\mathcal{E} = \emptyset$ and hence $\inf_{y \in A^{\circ} \cap \mathcal{E}}\Lambda^{\ast}(y) = \infty$. This proves the asserted statement.
\end{proof}

\subsection{Proof of Corollary \ref{cor: main_result}}

\begin{proof}[Proof of Corollary \ref{cor: main_result}]
 In view of Theorem \ref{main_result} and the G\"artner-Ellis Theorem \cite[Theorem 2.3.6.(c)]{dembo_zeitouni} it suffices to show that $\Lambda$ is essentially smooth and lower semi-continuous.
 
 \textit{$\Lambda$ is an essentially smooth function.} We have already shown that $\Lambda(\lambda) = F(y(\lambda)) < \infty$ for $\lambda < \lambda_c$. Moreover, $\Lambda$ is differentiable throughout $\mathcal{D}_{\Lambda}^{\circ}$ since $y(\lambda)$ is differentiable in $\lambda$, $y(\lambda) < y(\lambda_c) \leq \gamma_F$ and since the L\'evy measure $\nu$ appearing in the definition of $F$ has finite exponential moments up to order $\gamma_F$. It remains to show that $\Lambda$ is steep. Take the only boundary point of $\mathcal{D}_{\Lambda}^{\circ}$, i.e., $\lambda_{c}$. Let $(\lambda_{n})_n$ be a sequence with $\lambda_{n}\nearrow \lambda_{c}$. We obtain 
 \begin{align*}
  |\Lambda^\prime(\lambda_n)|
  = |y^\prime(\lambda_n)||F^\prime(y(\lambda_n))|
  = \frac{|F'(y(\lambda_n))|}{|R'(y(\lambda_n))|} \longrightarrow +\infty
 \end{align*}
 due to assumption \eqref{eq: boundary condition}.
 
 \textit{$\Lambda$ is a lower semicontinuous function.} We consider the different parts of the domain of $\Lambda$ separately. In each case, let $\{\lambda_n\}_n\subset\R$ be a sequence converging to $\lambda$. Our goal is to show that
 \begin{displaymath}
  \liminf_{n\to\infty}\Lambda(\lambda_n)\geq\Lambda(\lambda).
 \end{displaymath}
 For $\lambda>\lambda_c$, we have $\Lambda(\lambda)=\infty$ as well as $\Lambda(\lambda_n)=\infty$ for $n$ large enough. For $\lambda<\lambda_c$, we have $\Lambda(\lambda)=F(y(\lambda))$. As we have shown in Proposition \ref{prop: y properties}, this function is differentiable for such $\lambda$. Hence, it is continuous and therefore also lower semicontinuous.
 For $\lambda=\lambda_c$, note that if $\lambda_n>\lambda_c$, we have $\Lambda(\lambda_n)=\infty\geq\Lambda(\lambda_c)$.
 Hence, it suffices to consider sequences $\{\lambda_n\}$ with $\lambda_n\leq\lambda_c$ for all $n\in\N$. But then, we again have $\Lambda(\lambda_n)=F(y(\lambda_n))$ and $\Lambda(\lambda)=F(y(\lambda))$. Since $y(\lambda_n) \longrightarrow y(\lambda_c)=u_c$ and $F$ is continuous on $(-\infty,u_c]$ in the extended real-line with $F(y(\lambda_c)) \in [0,\infty]$, we conclude the assertion.
\end{proof}

\appendix

\section{Proofs of Lemma \ref{minimum_u_c} and Proposition \ref{prop: y properties}}
We start with the proof of Lemma \ref{minimum_u_c}.
\begin{proof}[Proof of Lemma \ref{minimum_u_c}]
 Note that $R^\prime(u)=\beta+\sigma^2u+\int_0^\infty(ze^{uz}-z)\mu(dz)$, and set 
\begin{equation}\label{r_prime_eqn}
 -\beta=\sigma^2u+\int_0^\infty(ze^{uz}-z)\mu(dz)=:\Phi(u).
\end{equation}
A possible minimum is the solution to equation \eqref{r_prime_eqn}. We consider two cases:
 
 \textit{Case 1:} $-\beta < \Phi(\gamma_R)\in(0,\infty]$: In this case, equation \eqref{r_prime_eqn} has a solution $u_c \in (0, \gamma_R)$ by continuity\footnote{Note that by monotone convergence, the case $\Phi(\gamma_R)=+\infty$ always implies that $\lim_{u\to\gamma_R}\Phi(u)=+\infty$.} and intermediate value theorem. Since $\Phi^\prime(u) = R''(u) > 0$ for all $u\in[0,\gamma_R)$ (see \eqref{eq: non degeneracy}), this is the only solution. Due to strict convexity of $R$, this is also a global minimum.
 
 \textit{Case 2:} $\Phi(\gamma_R)\leq-\beta$: In this case, we have 
 \begin{displaymath} 
  R^\prime(u)=\beta+\Phi(u) < \beta + \Phi(\gamma_R) \leq 0
 \end{displaymath} 
 for all $u\in(-\infty,\gamma_R)$ where the strict inequality follows from $\Phi'(u) > 0$. Therefore, $R$ is strictly decreasing and its minimum is attained at $u_c=\gamma_R$. ($R^\prime(\gamma_R)<\infty$ by assumption of this case.)

 For the last assertion, $\lambda_R > 0$, we note that by  $R(0)=0$ and since $u_c$ is the global minimum, we have $R(u_c)\leq R(0)=0$. Assuming that $R(u_c)=0$, we obtain from the strict convexity $R\big(\frac{1}{2}\cdot 0+\frac{1}{2}u_c\big)<\frac{1}{2}R(0)+\frac{1}{2}R(u_c)=0$ and hence $R(u_c/2)<0=R(u_c)$. This contradicts the fact that $u_c$ is a minimum. Thus we have shown that $R(u_c)<0$.
\end{proof}
Next we prove Proposition \ref{prop: y properties}.
\begin{proof}[Proof of Proposition \ref{prop: y properties}]
 First we have for all $u \in (- \infty, \gamma_R)$ and $\lambda > \lambda_R$
 \[
  R(u) + \lambda = R(u) + \lambda_R + (\lambda - \lambda_R) \geq \lambda - \lambda_R > 0
 \]
 and hence \eqref{equilibria} has no solution. However, for $\lambda \in (- \infty, \lambda_R]$, equation \eqref{equilibria} has at least one solution. Indeed, when $\lambda \in (0, \lambda_R]$ then noting that $R(0) + \lambda = \lambda > 0$ and $R(u_c) + \lambda = - \lambda_R + \lambda \leq 0$ shows that \eqref{equilibria} has a solution $y(\lambda) \in (0, u_c]$. When $\lambda = 0$, then we can take $y(\lambda) = 0$. Finally, for $\lambda < 0$ we have $R(0) + \lambda < 0$ and $R(u) \geq |\beta||u|$ so that $\lim_{u \to - \infty}R(u) = + \infty$ and hence \eqref{equilibria} has a solution $y(\lambda) \in (- \infty,0)$. Since $R$ is strictly convex, one can check that $y(\lambda)$ is the smallest solution of \eqref{equilibria}.
 
 Define the function $f(u,\lambda) = R(u) + \lambda$, where $u \in (-\infty, \gamma_R)$ and $\lambda < \lambda_R$. Then $f(y(\lambda), \lambda) = 0$. Moreover, since $y(\lambda)<u_c$ we obtain $\frac{\partial f(y(\lambda),\lambda)}{\partial u} \neq 0$ and the solution $y$ is continuously differentiable by the implicit function theorem. Differentiating $f(y(\lambda),\lambda)=0$ in $\lambda$ and solving for $y'$ gives \eqref{eq: y0 derivative}. 
 Since $\lim_{u \to - \infty}R(u) = +\infty$ and $R$ is strictly convex, it need to be decreasing in a neighborhood of $y(\lambda)$, i.e. $R'(y(\lambda)) < 0$. Thus $y'(\lambda) > 0$. Finally, we have $y(\lambda_R) = u_c > 0$. 
 
 It is left to show that $y$ is continuous in $\lambda_R$. Since $y(\lambda)<u_c$ for all $\lambda \in (0,\lambda_R)$, we obtain from \eqref{eq: y0 derivative} and $y(0) = 0$ that $y(\lambda) = \int_0^{\lambda} \frac{-1}{R'(y(\xi))}d\xi$. Since $R'(y(\xi)) < 0$, the function $y$ is strictly increasing and hence $y_R := \lim_{\lambda \nearrow \lambda_R}y(\lambda) \in [0,\infty]$ is well-defined. Since $y(\lambda) \leq y(\lambda_R) = u_c$ we conclude that $y_R < \infty$. By construction of $y(\lambda)$, we have $R(y(\lambda)) + \lambda = 0$. Since $R$ is continuous on $(-\infty, y_R] \subset (-\infty, u_c]$, we can pass to the limit and find that $R(y_R) + \lambda_R = 0$. Since $R(u) + \lambda_R \geq 0$ holds for each $u \in [0, \gamma_R)$ and $R$ is strictly convex, this equation has only one solution. Thus $y_R = y(\lambda_R)$. This proves the assertion.
\end{proof}
 
 \begin{figure}[!ht]
  \centering
  \begin{minipage}[b]{0.4\linewidth}
    \includegraphics[width=\linewidth]{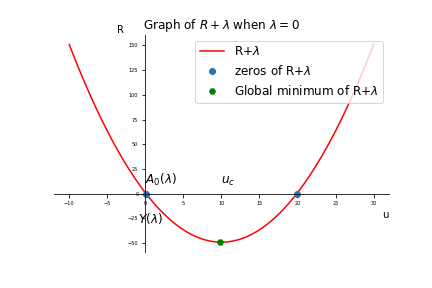}
  \end{minipage}
  \hfill
  \begin{minipage}[b]{0.4\linewidth}
    \includegraphics[width=\linewidth]{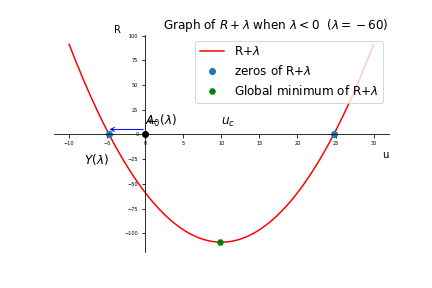}
  \end{minipage}
  \hfill
  \begin{minipage}[b]{0.4\linewidth}
    \includegraphics[width=\linewidth]{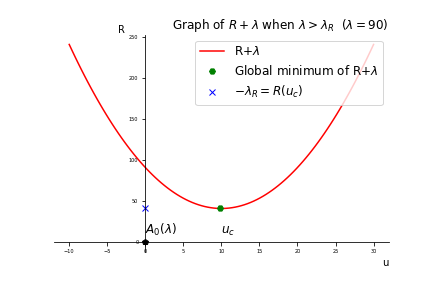}
  \end{minipage}
  \hfill
  \begin{minipage}[b]{0.4\linewidth}
    \includegraphics[width=\linewidth]{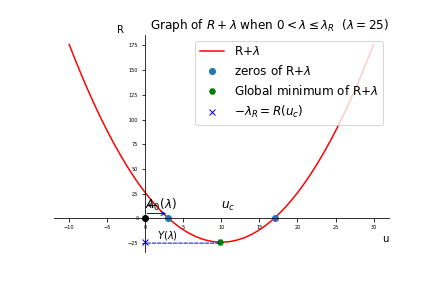}
  \end{minipage}
  \caption{The shape of $R(u)+\lambda$ as appearing on the right-hand side of \eqref{eq: A equation} as well as the corresponding behaviour of $A(t,\lambda)$ for the different regimes of $\lambda$.}
  \label{graphics_r}
\end{figure}

\section{Results from \cite{keller-ressel_mayerhofer}}

Below we state a special case of the results obtained in \cite{keller-ressel_mayerhofer}.
\begin{Theorem}[\cite{keller-ressel_mayerhofer}, Thm. 2.14 and Prop. 3.3]
Let $(X_t^x)$ be a CBI process on $\R_+$ and let $(Z_t)$ be defined as the integrated process
\begin{displaymath}
 Z_t=\int_0^t X_s^x ds.
\end{displaymath}
 Then the two-dimensional process $\xi=(X,Z)$ satisfies for $T\geq 0$ following statements:
\begin{enumerate}
    \item Let $\lambda\in\R^2$ and suppose that
    \begin{displaymath}
     \E^x[e^{\langle\lambda,\xi_T\rangle}] < \infty
    \end{displaymath}
    for some $x\in(0,\infty)^2$. Then there exists a unique solution $(v,w)$ up to time $T$ of the Riccati system
    \begin{equation}\label{appendix_riccati}
    \begin{split}
        \frac{\partial}{\partial t}v(t,\lambda)&=\bar{R}(v(t,\lambda)),\ v(0,\lambda)=\lambda
        \\
        \frac{\partial}{\partial t}w(t,\lambda)&=\bar{F}(v(t,\lambda)),\ w(0,\lambda)=0
    \end{split}
    \end{equation}
    such that
    \begin{equation}\label{affine_transformation}
     \E^x[e^{\langle\lambda,\xi_t\rangle}] = \exp(\langle x,v(t,\lambda)\rangle+w(t,\lambda))
    \end{equation}
    holds for all $x\in[0,\infty)$ and $t\in[0,T]$, where $\bar{R}\colon\R^2\to\R^2,\bar{F}\colon\R^2\to\R$ are given by
    \begin{displaymath}
     \bar{R}_1(u)=R(u_1)+u_2,\ \bar{R}_2(u)=0,\ \bar{F}(u)=F(u_1).
    \end{displaymath}
    \item Suppose that the Riccati system \eqref{appendix_riccati} has solutions $(v,w)$ starting at $(\lambda_1,\lambda_2,0)$ that exists up to time $T$. Then $\E^x[e^{\langle\lambda,\xi_T\rangle}]<\infty$. Additionally, the affine transformation formula \eqref{affine_transformation} holds.
\end{enumerate}
Furthermore, the quantity $T(x):=\sup\{t\geq 0\colon\E[e^{\langle\lambda,\xi_t\rangle}]<\infty\}$ is the maximum lifetime of the solution $(v,w)$.
\end{Theorem}

\begin{Remark}
To analyse the integrated CBI process, denoted by $Z$ in the above theorem, we consider the initial conditions $x=(x_1,0)$ as well as $\lambda=(0,\lambda_2)$.
\end{Remark}

\section{Boundedness of moments}

\begin{Lemma}\label{lemma: moment formula}
 Let $X$ be a CBI process with parameters $(b,\beta, \sigma, \nu, \mu)$ such that $\int_1^{\infty}z \nu(dz) < \infty$ holds. Then
 \[
  \E[X_t\ | \ X_s ] = X_s e^{\beta(t-s)} + m\left(1 - e^{\beta(t-s)}\right)
 \]
 holds for all $0 \leq s \leq t$.
\end{Lemma}
\begin{proof}
 Recall that by taking expectations in \eqref{SDE} and solving the corresponding ODE, we obtain 
 $\E[X_t] = \Big(x+\frac{\hat{b}}{\beta}\Big)e^{\beta t}-\frac{\hat{b}}{\beta} = xe^{\beta t} + m (1-e^{\beta t})$, see e.g. \cite{FJR2020b}. The conditional first moment can be now deduced from the the Markov property.
\end{proof}

\begin{Proposition}\label{prop: moment bounds}
 Let $X$ be a subcritical CBI process with parameters $(b,\beta, \sigma, \nu, \mu)$. Then the following assertions hold:
 \begin{enumerate}
     \item[(a)] If $\int_1^{\infty}z \nu(dz) < \infty$ holds, then
     \[
      \sup_{t \geq 0}\E[X_t^x] \leq \max\{x, m\}.
     \]
     
     \item[(b)] If \eqref{eq: second moments} holds, then
     \[
      \sup_{t \geq 0}\E[(X_t^x)^2] < \infty.
     \]
 \end{enumerate}
\end{Proposition}
\begin{proof}
 (a) The assertion follows from 
 \[
  0 \leq \E[X_t^x] = xe^{\beta t} + m(1-e^{\beta t}) \leq \max\{x, m\}.
 \]
 (b) For the second assertion, we use the same representation as in \cite{FJR2020b} to write
 \begin{align*}
     X_t^x &= e^{\beta t}x + m\left( 1- e^{\beta t}\right) + \sigma \int_0^t e^{\beta (t-s)}\sqrt{X_s^x}dB_s
     \\ &\qquad + \int_0^t \int_0^{\infty}\int_0^{\infty} e^{\beta(t-s)}z\1_{\{ X_{s-}\leq u\}} \widetilde{N}_{\mu}(ds,dz,du)
     + \int_0^t \int_0^{\infty}e^{\beta(t-s)}z \widetilde{N}_{\nu}(ds,dz).
 \end{align*}
 Thus we find for some generic constant $C > 0$
 \begin{align*}
     \E[(X_t^x)^2] &\leq C\left( e^{\beta t}x + m\left( 1- e^{\beta t}\right) \right)^2
     + C \E\left[ \left| \sigma \int_0^t e^{\beta(t-s)}\sqrt{X_s^x}dB_s \right|^2 \right]
     \\ &\qquad + C\E\left[ \left| \int_0^t \int_0^{\infty}\int_0^{\infty} e^{\beta(t-s)}z\1_{\{ X_{s-}\leq u\}} \widetilde{N}_{\mu}(ds,dz,du) \right|^2 \right]
     \\ &\qquad + C\E\left[\left| \int_0^t \int_0^{\infty}e^{\beta(t-s)}z \widetilde{N}_{\nu}(ds,dz) \right|^2 \right] 
     \\ &\leq C + C \int_0^t e^{2\beta(t-s)}\E[X_s^x]ds + C\int_0^t \int_0^{\infty} z^2 e^{2\beta(t-s)} \E[X_s^x] ds \mu(dz)
     \\ &\qquad + C\int_0^t \int_0^{\infty} z^2 e^{2\beta(t-s)}\E[X_s^x] ds \nu(dz)
     \\ &\leq C\left( 1  + \int_0^{\infty}z^2 \mu(dz) + \int_0^{\infty} z^2 \nu(dz) \right) < \infty.
 \end{align*}
\end{proof}

\bibliographystyle{amsplain}
\phantomsection\addcontentsline{toc}{section}{\refname}\bibliography{Bibliography}

\end{document}